\documentclass[12pt,oneside,reqno]{amsart}
\usepackage{amsfonts,amsmath,amscd,amssymb,paralist}
\usepackage{mathrsfs}
\usepackage{amsxtra}
\usepackage[mathscr]{eucal}
\usepackage{graphics}
\usepackage{latexsym,graphicx,verbatim,nameref}
\usepackage[numbers,sort&compress]{natbib}
\usepackage[all,cmtip]{xy}
\usepackage{hyperref}
\usepackage{fancyhdr}
\usepackage{color}
\usepackage{tikz}
\usepackage{tikz-cd}
\usepackage{graphicx}
\usetikzlibrary{arrows,scopes,matrix}
\numberwithin{equation}{section}

\theoremstyle{definition}

\makeatletter
\newtheorem*{rep@theorem}{\rep@title}
\newcommand{\newreptheorem}[2]{%
\newenvironment{rep#1}[1]{%
 \def\rep@title{#2 \ref{##1}}%
 \begin{rep@theorem}}%
 {\end{rep@theorem}}}
\makeatother

\newreptheorem{theorem}{Theorem}
\newreptheorem{corollary}{Corollary}
\newreptheorem{lemma}{Lemma}
\newreptheorem{conjecture}{Conjecture}

\newtheorem{theorem}{Theorem}[section]
\newtheorem{corollary}[theorem]{Corollary}

\newtheorem{proposition}[theorem]{Proposition}
\newtheorem*{theorem*}{Theorem}
\newtheorem*{proposition*}{Proposition}
\newtheorem{definition}[theorem]{Definition}
\newtheorem{remark}[theorem]{Remark}

\newtheorem{example}[theorem]{Example}

\newtheorem*{claim*}{Claim}

\newtheorem*{conjecture*}{Conjecture}
\newtheorem*{observation*}{Observation}
\newtheorem*{question*}{Question}

\begin{document}
\title{A Wiener-type theorem for arcs in the unit circle}
\author{Huichi Huang}
\address{Huichi Huang, Mathematical Department, College of Mathematics and Statistics, Chongqing University, Chongqing, 401331, PR China}
\thanks{The author was partially supported by NSFC no. 11871119 and Chongqing Municipal Science and Technology Commission fund no. cstc2018jcyjAX0146.}
\email{huanghuichi@cqu.edu.cn}
\keywords{Fej\'er's theorem, Wiener's theorem}
\subjclass[2010]{Primary: 37A05; Secondary: 11J71}
\date{\today}
\begin{abstract}
  We prove a Wiener-type theorem for arcs in the unit circle which  expresses the measure of an arc in the unit circle via the measure's Fourier coefficients. Then we use it to give the Fourier series of the Cantor function and to compute the local dimension of a measure satisfying certain conditions of Fourier coefficients.
\end{abstract}

\maketitle
\tableofcontents
\section{Introduction}

A classical result in harmonic analysis, due to N. Wiener in 1933, says that  on the unit circle $\mathbb{T}$, one can express a finite Borel measure's value at a single point  via this measure's Fourier coefficients~\cite{Wiener1933}~\cite[Chap. I, 7.13]{Katznelson2004}.
\begin{theorem*}~[Wiener's theorem]\

Suppose $\mu$ is a finite Borel measure on $\mathbb{T}$. Then for every $z\in\mathbb{T}$,
$$\mu\{z\}=\lim_{N\to\infty}\frac{1}{2N+1}\sum_{n=-N}^N \hat{\mu}(n)z^n.$$
\end{theorem*}

In this article we prove a Wiener-type theorem for arcs, which is to say,  an arc's measure can also be expressed in terms of the measure's Fourier coefficients.

\begin{theorem}~\label{thm:arc}~[Wiener-type theorem for arcs]\

Suppose $\mu$ is a Borel probability measure on $\mathbb{T}$, then

\begin{align*}
\mu[a,b)&=\frac{1}{\pi}\sum_{n=1}^\infty \Im(\frac{\hat{\mu}(n)-1}{n}[e(nb)-e(na)])    \\
&+\lim_{n\to \infty}\frac{1}{2n}\sum_{j=1}^n (\hat{\mu}(j)-1)[e(ja)-e(jb)]
\end{align*}
for all $0<a<b\leq 1$. Moreover
\begin{align*}
\mu[0,b)&=1+\frac{1}{\pi}\sum_{n=1}^\infty \Im(\frac{\hat{\mu}(n)-1}{n}[e(nb)-1])   \\
&+\lim_{n\to \infty}\frac{1}{2n}\sum_{j=1}^n (\hat{\mu}(j)-1)[1-e(jb)]
\end{align*}
for every  $0<b\leq 1$.
The notation  $\Im(a)$ stands for the imaginary part of a complex number $a$ and $e(x)=\exp(2\pi i x)$ for a real number $x$.
\end{theorem}

\section{Preliminary}

Identify $\mathbb{T}$ with $[0,1)$ when necessary.

 For a finite Borel measure on $\mathbb{T}$, define the {\bf distribution function} $D_\mu(x)$ by
$D_\mu(x)=\mu[0,x)$ for every $x\in[0,1]$.

The {\bf Fourier coefficients} $\{\hat{\mu}(n)\}_{n\in\mathbb{Z}}$ of a finite Borel measure $\mu$ on $\mathbb{T}$ are given by
$\hat{\mu}(n)=\int_{\mathbb{T}} e(-nx)\,d\mu(x).$

For $f\in L^p(\mathbb{T})$ for $p\in(1,\infty)$, the {\bf Fourier series} of $f$ is
$\sum_{n\in\mathbb{Z}} \hat{f}(n)e(nx),$ hereafter $e(x)$ stands for $e^{2\pi ix}$ for all $x\in \mathbb{R}$ and $\hat{f}(n)=\int_0^1 f(t)e(-nt)\,dt$ for all $n\in\mathbb{Z}$.

For $f,g\in L^1(\mathbb{T})$, the {\bf convolution} $f*g$ of $f$ and $g$ is given by
$$f*g(x)=\int_{0}^{1}f(y)g(x-y)\,dy$$ for all $x$ in $[0,1)$. Here $x-y$ is interpreted as the fractional part of $x-y$ which is in $[0,1)$.

Denote $\sum_{n=-N}^N\hat{f}(n)e(nx)$ by $S_N(f)(x)$.  By Carleson-Hunt theorem~\cite{Carleson1966,Hunt1968}, $S_n(f)(x)$ converges to $f(x)$ almost everywhere, but not  pointwisely even when $f$ is continuous~\cite[Chap. II, Sec. 2]{Katznelson2004}.


The {\bf Fej\'er kernel} for $\mathbb{T}$ is given by
$$K_n(t)=\sum_{j=-n}^n (1-\frac{|j|}{n+1})e(jt)=\frac{\sin^2\pi t}{(n+1)\sin^2{(n+1)\pi t}}$$ for every nonnegative integer $n$.

Fej\'er kernel has many nice properties. Below we list some of them which would be frequently used throughout the paper. See~\cite[p.181, Prop. 3.1.10]{Grafakos2014} and~\cite[p.205, (3.4.3)]{Grafakos2014} for proof.
\begin{proposition}~[Properties of $K_n$]~\label{pcfejer}
\begin{enumerate}
\item $K_n(t)=K_n(-t)$ for every $t\in \mathbb{T}$.
\item $K_n(t)\geq 0$ for all $t\in\mathbb{T}$.
\item $\int_{\mathbb{T}} K_n(t)\, dt=1$.
\item For $\lambda\in(0,\frac{1}{2})$, we have $\displaystyle\sup_{t\in [\lambda,1-\lambda]} {K_n(t)}\leq\frac{1}{(n+1)\sin^2\pi\lambda},$   hence $\displaystyle\lim_{n\to\infty}\int_{\lambda}^{1-\lambda} K_n(t)\,dt=0.$
\end{enumerate}
\end{proposition}
Fej\'er's theorem is as follows~\cite[Chap. I, Thm. 3.1(a)]{Katznelson2004}~\cite[Thm. 3.4.1]{Grafakos2014}.
\begin{theorem*}~[Fej\'{e}r's theorem]
~\label{Fejer1}\

If $f\in L^1(\mathbb{T})$ and for some $t\in[0,1)$, both the left and right limit of $f$ exist~(denoted by $f(t-)$ and $f(t+)$ respectively ), then
$$\displaystyle\lim_{n\to\infty}\sigma_n(f,t)=\frac{f(t+)+f(t-)}{2}.$$ Here
 $\sigma_n(f,t)=K_n*f(t)=\sum_{j=-n}^n (1-\frac{|j|}{n+1})\hat{f}(j)e(jt)$ for all nonnegative integer $n$.
\end{theorem*}

We also need a  generalized version of Wiener's theorem, which is a special case of~\cite[Thm. 4.1]{Huang2016}.

\begin{theorem}~[Generalized Wiener's theorem]\

Suppose $\mu$ is a finite Borel measure on $\mathbb{T}$. Then for every $z\in\mathbb{T}$, we have
$$\mu\{z\}=\lim_{n\to\infty}\frac{1}{|F_n|}\sum_{j\in F_n} \hat{\mu}(j)z^j.$$ Here $\{F_n\}_{n=1}^\infty$ is an arbitrarily chosen F{\o}lner sequence of $\mathbb{Z}$ and $|F_n|$ is the cardinality of $F_n$.
\end{theorem}

\section{A Wiener-type theorem for arcs}\

In this section, we prove the following.

\begin{theorem}~[a Wiener-type theorem for arcs]~\label{thm:arc}\

Suppose $\mu$ is a Borel probability measure on $\mathbb{T}$, then for all $0<a<b\leq 1$, we have

\begin{align}~\label{eq:wtarc}
\mu[a,b)&=\frac{1}{\pi}\sum_{n=1}^\infty \Im(\frac{\hat{\mu}(n)-1}{n}[e(nb)-e(na)]) \notag \\
&+\lim_{n\to \infty}\frac{1}{2n}\sum_{j=1}^n (\hat{\mu}(j)-1)[e(ja)-e(jb)].
\end{align}
Moreover for every  $0<b\leq 1$, we have
\begin{align}~\label{eq:wtarc2}
\mu[0,b)&=1+\frac{1}{\pi}\sum_{n=1}^\infty \Im(\frac{\hat{\mu}(n)-1}{n}[e(nb)-1])   \notag \\
&+\lim_{n\to \infty}\frac{1}{2n}\sum_{j=1}^n (\hat{\mu}(j)-1)[1-e(jb)].
\end{align}
The notation  $\Im(a)$ stands for the imaginary part of a complex number $a$.
\end{theorem}

\begin{proof}

Let $D_\mu(x)$ be $\mu[0,x)$ for every $x\in[0,1]$. Denote $D_\mu(x)-x$ by $f$. Note that both $D_\mu(x)$ and $x$ are monotone functions, so $f$ is of bounded variation. Hence integration by parts works for $f$.

Then for every nonzero $n$,  we have
\begin{align*}
&\hat{f}(n)=\int_0^1 f(x)e(-nx)\,dx=\int_0^1 f(x)\,d[\frac{1}{-2\pi in}e(-nx)] \\
&=\frac{1}{-2\pi in}[D_\mu(x)-x]e(-nx)|_0^1+\frac{1}{2\pi in} \int_0^1 e(-nx)\,d [D_\mu(x)-x]   \\
&=\frac{1}{2\pi in} \int_0^1 e(-nx)\,d\mu(x)=\frac{\hat{\mu}(n)}{2\pi in},
\end{align*}
and $\hat{f}(0)=\int_0^1 f(x)\,dx.$

Hence
\begin{align*}
&\sigma_n(f,t)=\sum_{j=-n}^n \hat{f}(j)e(jt)-\sum_{j=-n}^n \frac{|j|}{n+1}\hat{f}(j)e(jt)\\
&=S_n(f)(t)-\sum_{j=1}^n \frac{j}{n+1}\frac{\hat{\mu}(j)}{2\pi ij}e(jt)-\sum_{j=-n}^{-1} \frac{-j}{n+1}\frac{\hat{\mu}(j)}{2\pi ij}e(jt)\\
&=S_n(f)(t)-\sum_{j=1}^n \frac{j}{n+1}[\frac{\hat{\mu}(j)}{2\pi ij}e(jt)+\frac{\hat{\mu}(-j)}{-2\pi ij}e(-jt)]    \\
&=S_n(f)(t)-\frac{1}{n+1}\sum_{j=1}^n\frac{1}{2\pi i} [\hat{\mu}(j)e(jt)-\overline{\hat{\mu}(j)e(jt)}] \\
&=S_n(f)(t)-\frac{1}{\pi}\frac{1}{n+1}\sum_{j=1}^n \Im(\hat{\mu}(j)e(jt)).
\end{align*}
Since $\mu$ is a Borel probability measure, it follows from Wiener's theorem~\footnote{We use a version of Wiener's theorem different from~\cite[Chap. I, 7.13]{Katznelson2004}, and the difference lies in choices of F{\o}lner sequences of $\mathbb{Z}$. But Wiener's theorem holds for every F{\o}lner sequence. See~\cite{AnoussisBisbas2000} and~\cite[Thm. 4.1]{Huang2016}.} that
$$\lim_{n\to\infty}\frac{1}{n+1}\sum_{j=1}^n \Im(\hat{\mu}(j)e(jt))=\Im(\lim_{n\to\infty}\frac{1}{n+1}\sum_{j=1}^n \hat{\mu}(j)e(jt))=\Im(\mu\{t\})=0.$$  Hence by Fej\'{e}r's theorem, for every $0<t<1$ we have that
$$\frac{f(t+)+f(t-)}{2}=\lim_{n\to\infty}\sigma_n(f,t)=\lim_{n\to\infty} S_n(f)(t).$$

Since $f$ is  real-valued, it is easy to see that $\overline{\hat{f}(n)}=\hat{f}(-n)$.

We have
$$S_n(f)(t)=\frac{1}{\pi}\sum_{j=1}^n \frac{1}{j}\Im(\hat{\mu}(j)e(jt))+\hat{f}(0).$$

Since $f(0+)=\mu\{0\}$ and $f(0-)=f(1-)=\displaystyle\lim_{x\to 1-}\mu[0,x)-x=1-1=0$, we get
\begin{align*}
&\hat{f}(0)=\frac{\mu\{0\}}{2}-\frac{1}{\pi}\sum_{j=1}^\infty \Im(\frac{\hat{\mu}(j)}{j})\\
&=\lim_{n\to \infty}\frac{1}{2n}\sum_{j=1}^n \hat{\mu}(j)-\frac{1}{\pi}\sum_{j=1}^n \Im(\frac{\hat{\mu}(j)}{j}).
\end{align*}
For every $x\in [0,1]$, the limits $f(x+)$ and $f(x-)$ always exist and equal $\mu[0,x)+\mu[0,x]-2x$.

Hence
\begin{align*}
&\frac{\mu[0,x)+\mu[0,x]}{2}-x=\frac{f(x+)+f(x-)}{2}  \\
&=\frac{1}{\pi}\sum_{n=1}^\infty\Im(\frac{\hat{\mu}(n)}{n}e(nx))+\hat{f}(0)  \\
&=\frac{1}{\pi}\sum_{n=1}^\infty\Im(\frac{\hat{\mu}(n)}{n}[e(nx)-1])+\lim_{n\to \infty}\frac{1}{2n}\sum_{j=1}^n \hat{\mu}(j)
\end{align*}
for every $x$ in $[0,1]$.

Note that $\frac{\mu[0,x)+\mu[0,x]}{2}=D_\mu(x)+\frac{\mu\{x\}}{2}$ and it follows from Wiener's theorem again that $\mu\{x\}\displaystyle=\lim_{n\to\infty}\frac{1}{n}\sum_{j=1}^n\hat{\mu}(j)e(jx)$.

Hence
\begin{equation}~\label{eq: comparison}
D_\mu(x)-x=\frac{1}{\pi}\sum_{n=1}^\infty\Im(\frac{\hat{\mu}(n)}{n}[e(nx)-1])+\lim_{n\to \infty}\frac{1}{2n}\sum_{j=1}^n \hat{\mu}(j)[1-e(jx)].
\end{equation}

Let $\nu$ be another Borel probability measure on $\mathbb{T}$.

Then we have
\begin{align*}
D_\mu(x)-D_\nu(x)&=\frac{1}{\pi}\sum_{n=1}^\infty\Im(\frac{\hat{\mu}(n)-\hat{\nu}(n)}{n}[e(nx)-1])   \\
&+\lim_{n\to \infty}\frac{1}{2n}\sum_{j=1}^n (\hat{\mu}(j)-\hat{\nu}(j))[1-e(jx)].
\end{align*}

We choose $\nu=\delta_0$~(the Dirac measure concentrating at the point 0). Note that $D_{\delta_0}(x)=1$ for $0<x\leq 1$ and $\hat{\delta_0}(n)=1$ for every  integer $n$.
Hence for every  $0<b\leq 1$, one has
\begin{align*}~\label{eq:wtarc2}
\mu[0,b)&=1+\frac{1}{\pi}\sum_{n=1}^\infty \Im(\frac{\hat{\mu}(n)-1}{n}[e(nb)-1])   \notag \\
&+\lim_{n\to \infty}\frac{1}{2n}\sum_{j=1}^n (\hat{\mu}(j)-1)[1-e(jb)].
\end{align*}

For all $0<a<b\leq 1$, since $\mu[a,b)=D_\mu(b)-D_\mu(a)$,  we get

\begin{align*}
\mu[a,b)&=\frac{1}{\pi}\sum_{n=1}^\infty \Im(\frac{\hat{\mu}(n)-1}{n}[e(nb)-e(na)]) \notag \\
&+\lim_{n\to \infty}\frac{1}{2n}\sum_{j=1}^n (\hat{\mu}(j)-1)[e(ja)-e(jb)].
\end{align*}

\end{proof}

\begin{corollary}~\label{cor:d}
If $\mu$ is a  Borel probability measure on $\mathbb{T}$, then
\begin{equation}~\label{eq:d}
\mu[a,b)-(b-a)=\frac{1}{\pi}\sum_{n=1}^\infty  \Im(\frac{\hat{\mu}(n)}{n}[e(nb)-e(na)])+\lim_{n\to \infty}\frac{1}{2n}\sum_{j=1}^n \hat{\mu}(j)[e(ja)-e(jb)].
\end{equation}
for all $0<a<b\leq 1$.
\end{corollary}

\begin{corollary}~\label{cor:wtc}
If $\mu$ is a continuous measure on $\mathbb{T}$, then
\begin{equation}~\label{eq:wtc}
\mu[a,b)=\frac{1}{\pi}\sum_{n=1}^\infty \Im(\frac{\hat{\mu}(n)-1}{n}[e(nb)-e(na)])
\end{equation}
for all $0<a<b<1$, and
\begin{equation}~\label{eq:wtc2}
\mu[0,b)=\frac{1}{2}+\frac{1}{\pi}\sum_{n=1}^\infty \Im(\frac{\hat{\mu}(n)-1}{n}[e(nb)-1])
\end{equation}
for every $0<b<1$.
\end{corollary}
\begin{proof}
Equations~\ref{eq:wtc} and~\ref{eq:wtc2} are special cases of  Equations~\ref{eq:wtarc} and~\ref{eq:wtarc2}. Combining Wiener's theorem and Theorem~\ref{thm:arc} gives the proof.
\end{proof}

Denote the set of complex-valued continuous functions on $\mathbb{T}$ by $C(\mathbb{T})$.

The {\bf convolution} $\mu*\nu$ of two finite Borel measures $\mu$ and $\nu$ is a finite Borel measure given by
$$\mu*\nu(f)=\int_0^1\int_0^1f(x+y)\,d\mu(x)d\nu(y)$$ for all $f\in C(\mathbb{T})$. From definition one has
$\widehat{\mu*\nu}(n)=\hat{\mu}(n)\hat{\nu}(n)$ for all $n\in\mathbb{Z}$.

The {\bf conjugate} $\bar{\mu}$ of a finite Borel measure $\mu$ is a finite Borel measure given by
$$\bar{\mu}(f)=\int_0^1 f(1-x)\,d\mu(x)$$ for all $f\in C(\mathbb{T})$. Note that $\hat{\bar{\mu}}(n)=\overline{\hat{\mu}(n)}$ for all $n\in\mathbb{Z}$.

As a consequence of Theorem~\ref{thm:arc}, we have the following.

\begin{corollary}~\label{conv}
Suppose $\mu$ is a continuous Borel probability measure on $\mathbb{T}$, then
\begin{equation}~\label{eq:wt-conv}
\mu*\bar{\mu}[a,b)=\frac{1}{\pi}\sum_{n=1}^\infty \frac{|\hat{\mu}(n)|^2-1}{n}[\sin(2\pi nb)-\sin(2\pi na)]
\end{equation}
for all $0<a<b\leq 1$, and
\begin{equation}~\label{eq:wt-conv2}
\mu*\bar{\mu}[0,b)=\frac{1}{2}+\frac{1}{\pi}\sum_{n=1}^\infty  \frac{|\hat{\mu}(n)|^2-1}{n}\sin(2\pi nb)+\lim_{n\to \infty}\frac{1}{2n}\sum_{j=1}^n e(jb)
\end{equation}
for every $0<b\leq 1$.
In particular, for all $0<b<1$, it holds that
\begin{equation}~\label{eq:wt-conv2}
\mu*\bar{\mu}[0,b)=\frac{1}{2}+\frac{1}{\pi}\sum_{n=1}^\infty  \frac{|\hat{\mu}(n)|^2-1}{n}\sin(2\pi nb).
\end{equation}
\end{corollary}
\begin{proof}
It follows from that $\widehat{\mu*\bar{\mu}}(n)=|\hat{\mu}(n)|^2$ for all $n\in\mathbb{Z}$.
\end{proof}

\begin{remark}
  Let $\mu$ be the Lebesgue measure on $\mathbb{T}$. Then $\mu$ is a continuous measure with $\hat{\mu}(n)=0$ for $n\geq 1$. Applying Equation~\ref{eq:wt-conv2}, we get
  $$\sum_{n=1}^{\infty}\frac{\sin nx}{n}=\frac{\pi-x}{2}$$ for all $x\in (0,2\pi)$.
\end{remark}

\section{Applications}
\subsection{Fourier series of the Cantor function}\

Let $C(x)$ be the Cantor function, i.e., the distribution function of the Cantor measure $c$. It's known that~\cite{BlumEpstein1974}
$$\hat{c}(n)=(-1)^n\prod_{k=1}^\infty \cos(\frac{2n\pi}{3^k}).$$

Then we get the Fourier series of $C(x)$:
$$C(x)=\frac{1}{2}+\frac{1}{\pi}\sum_{n=1}^\infty \frac{(-1)^n\prod_{k=1}^\infty \cos(\frac{2n\pi}{3^k})-1}{n}\sin(2\pi nx)$$
for every $0<x<1$.
\subsection{Computations of local dimensions}\

In this section, we only talk about measures on the unit circle.


\begin{definition}~\label{def: locDim}~\cite[10.1]{Falconer1997}\

The  lower and upper {\bf local dimension} of a measure $\mu$ on $\mathbb{T}$ at a point $x$  are given by

$$\underline{\dim}(\mu,x)=\underset{r\to 0+}{\underline{\lim}}\frac{\log \mu(B_r(x))}{\log r},$$ and

$$\overline{\dim}(\mu,x)=\underset{r\to 0+}{\overline{\lim}}\frac{\log \mu(B_r(x))}{\log r}.$$ We say that the {\bf local dimension} of $\mu$, denoted by $\dim_{\rm loc}(\mu,x)$, exists at $x$ if these two coincide. Here $B_r(x)$ is the open ball centered at $x$ with radius $r$.
\end{definition}
If $x$ is not in the support of $\mu$, then $\dim_{\rm loc}(\mu,x)=\infty$. If $x$ is an atom for $\mu$, then $\dim_{\rm loc}(\mu,x)=0$.

When a measure satisfies certain conditions,  Equation~\ref{eq: comparison} helps in computations of  $\dim_{\rm loc}(\mu,x)$.
\begin{theorem}~\label{th: locDim}
If $\mu$ is a Borel probability measure such that $\sum_{n=1}^\infty |\hat{\mu}(n)|<\frac{1}{2}$, then $\dim_{\rm loc}(\mu,x)=1$ for every $x\in (0,1)$.
\end{theorem}
\begin{proof}
By Wiener's theorem $\mu$ is a continuous measure. By Equation~\ref{eq:wtc}, for every $x\in(0,1)$, we have
$$\mu(B_r(x))=2r+\frac{1}{\pi}\sum_{n=1}^\infty \Im[\frac{\hat{\mu}(n)}{n}(e(n(x+r))-e(n(x-r)))].$$
Moreover
\begin{align*}
&|\frac{1}{2\pi r}\sum_{n=1}^\infty \Im[\frac{\hat{\mu}(n)}{n}(e(n(x+r))-e(n(x-r)))]| \\
&\leq \frac{1}{2\pi r}\sum_{n=1}^\infty |\frac{\hat{\mu}(n)}{n}(e(n(x+r))-e(n(x-r)))| \\
&=\frac{1}{2\pi r}\sum_{n=1}^\infty |\frac{\hat{\mu}(n)}{n}(e(nr)-e(-nr))| \\
&\leq \sum_{n=1}^\infty |\hat{\mu}(n)|\frac{2|\sin (2n\pi r)|}{2n\pi r}
\tag{$\frac{|\sin x|}{|x|}\leq 1$ for all nonzero $x$} \\
&\leq 2\sum_{n=1}^\infty |\hat{\mu}(n)|<1.
\end{align*}
for all $r>0$.

It follows that
\begin{align*}
&\log(1-2\sum_{n=1}^\infty |\hat{\mu}(n)|)\leq \log(1+\frac{1}{2\pi r}\sum_{n=1}^\infty \Im[\frac{\hat{\mu}(n)}{n}(e(n(x+r))-e(n(x-r)))])  \\
&\leq \log(1+2\sum_{n=1}^\infty |\hat{\mu}(n)|)
\end{align*}
for all $r>0$.

Hence
\begin{align*}
&\frac{\log \mu(B_r(x))}{\log r}=\frac{1}{\log r}[\log 2r+\log(1+\frac{1}{2\pi r}\sum_{n=1}^\infty \Im[\frac{\hat{\mu}(n)}{n}(e(n(x+r))-e(n(x-r)))])]   \\
&\to 1
\end{align*}
as $r\to 0$.
\end{proof}

We construct examples of measures satisfying the assumption of Theorem~\ref{th: locDim}.

\begin{example}
Take a sequence of real numbers $\{a_n\}_{n=1}^\infty$ such that $\sum_{n=1}^\infty |a_n|<1$. Define $\mu$ by $\mu(f)=\int_0^1 f(t)[1+\sum_{n=1}^\infty a_n\cos(2\pi nt)]\,dt$ for all $f\in C(\mathbb{T})$. Then $\mu$ is a Borel probability measure on $\mathbb{T}$ such that $\hat{\mu}(n)=\frac{a_n}{2}$ for all $n\geq 1$.
\end{example}

\end{document}